\documentclass[11pt]{article}
\usepackage[T1]{fontenc}
\usepackage{lmodern}
\usepackage[margin=1in]{geometry}
\usepackage{amsmath,amssymb,amsthm,mathtools}
\usepackage{microtype}
\usepackage{booktabs}
\usepackage{needspace}
\usepackage[numbers,sort&compress]{natbib}
\usepackage[hidelinks]{hyperref}
\usepackage[nameinlink,noabbrev]{cleveref}
\hypersetup{pdftitle={Linear stability and divisibility gaps for Gaussian maxima},
 pdfauthor={Chahel Paatur}}
\numberwithin{equation}{section}
\newtheorem{theorem}{Theorem}[section]
\newtheorem{proposition}{Proposition}[section]
\newtheorem{lemma}{Lemma}[section]
\newtheorem{corollary}{Corollary}[section]
\theoremstyle{definition}

\theoremstyle{remark}
\newtheorem{remark}{Remark}[section]
\newcommand{\R}{\mathbb R}
\newcommand{\E}{\mathbb E}
\newcommand{\Prob}{\mathbb P}
\newcommand{\ip}[2]{\langle #1,#2\rangle}
\newcommand{\norm}[1]{\lVert #1\rVert}
\newcommand{\supp}{\operatorname{supp}}
\newcommand{\conv}{\operatorname{conv}}
\newcommand{\tr}{\operatorname{tr}}
\newcommand{\spec}{\operatorname{spec}}
\newcommand{\diag}{\operatorname{diag}}
\newcommand{\dist}{\operatorname{dist}}
\newcommand{\op}{\mathrm{op}}
\newcommand{\Phibar}{\overline\Phi}
\newcommand{\cA}{\mathcal A}
\newcommand{\cR}{\mathcal R}
\newcommand{\cF}{\mathcal F}
\newcommand{\basis}{\mathcal B}
\newcommand{\ind}{\mathbf 1}
\newenvironment{proofof}[1]{\begin{proof}[Proof of #1]}{\end{proof}}
\title{Linear stability and divisibility gaps\\for Gaussian maxima}
\author{Chahel Paatur\thanks{\href{mailto:chahelpaatur@gmail.com}{chahelpaatur@gmail.com}}\\[3pt]
\small University of Illinois Urbana-Champaign\\
\small Urbana, Illinois, USA}
\date{}
\begin{document}
\maketitle

\begin{abstract}
For a compact set of unit vectors carrying an isotropic measure, we prove
that its Gaussian mean width minus the expected maximum $a_k$ of $k$
independent standard normals is comparable to its Hausdorff distance from
an orthonormal basis. The constants depend only on the dimension, with no
restriction on the number of directions or their weights. Applied to
correlation matrices of size $N$ with $k$ equal nonzero eigenvalues, this
gives a minimum of $a_k$ when $k$ divides $N$ and an excess of order
$N^{-1/2}$ otherwise. In dimension two we determine the leading term.
For odd $N$, the minimum is
$1/\sqrt\pi+1/(2\sqrt{\pi N})+O(N^{-1})$.
An asymmetric three-direction frame attains this expansion. We also
characterize the scaled angular distributions of asymptotic minimizers,
prove covariance-block recovery, and derive a finite-dimensional lower
bound for the limiting problem in higher rank.
\end{abstract}



\section{Introduction and main results}

Let $g$ be a standard Gaussian vector in $\R^k$, where $k\ge2$, and let
$S\subset S^{k-1}$ be nonempty and compact. Define
\[
 h_S(x)=\sup_{u\in S}\ip{x}{u},\qquad
 w(S)=\E h_S(g),\qquad
 a_k=\E\max_{1\le j\le k}g_j.
\]
We use the Gaussian normalization of mean width. In particular, $w(S)$ depends
only on $\conv S$, although the distance in our stability theorem concerns the
set $S$ itself. The assumption $S\subset S^{k-1}$ implies
$|h_S(g)|\le\norm{g}$, so all these expectations are finite.

We call $S$ \emph{admissible} if there is a finite positive Borel measure $\mu$
supported in $S$ such that
\begin{equation}\label{eq:isotropy}
 \int_S uu^{\mathsf T}\,d\mu(u)=I_k.
\end{equation}
Taking traces gives $\mu(S)=k$. We do not require $\supp\mu=S$. This convention
allows directions of vanishing weight in limits and makes the class of admissible
sets compact. Write $\cA_k$ for this class. Every finite weighted isotropic family
$(u_i,d_i)_{i=1}^N$, with $\norm{u_i}=1$, $d_i>0$, and
$\sum_i d_i u_i u_i^{\mathsf T}=I_k$, gives an admissible set.

Let $\basis_k=\{e_1,\ldots,e_k\}$ and put
\begin{equation}\label{eq:distances}
 \begin{split}
 d_H(S,T)&=\max\left\{\sup_{u\in S}\dist(u,T),
                       \sup_{v\in T}\dist(v,S)\right\},\\
 \delta(S)&=\min_{Q\in O(k)}d_H(S,Q\basis_k),\qquad
 \Delta(S)=w(S)-a_k,\qquad r_k=\E\norm{g}.
 \end{split}
\end{equation}
The distance is Euclidean. The minimum exists because $O(k)$ is compact. An
orthonormal basis here consists of $k$ oriented unit vectors; it does not include
both signs of each vector. Also $r_k\le\sqrt{k}$ by Jensen's inequality.

The geometric Brascamp--Lieb inequality implies
\begin{equation}\label{eq:dom-intro}
 \Prob\{h_S(g)\le t\}\le\Phi(t)^k\qquad(t\in\R),
\end{equation}
where $\Phi$ is the standard Gaussian distribution function. In particular,
$\Delta(S)\ge0$. Equality in the expectation holds precisely when $S$ is an
orthonormal basis. Both statements are proved in
\cref{sec:preliminaries}. Our first result quantifies equality.

\begin{theorem}[Linear stability]\label{thm:stability}
For every $k\ge2$ there is a constant $c_k>0$ such that every $S\in\cA_k$ satisfies
\begin{equation}\label{eq:stability}
 c_k\delta(S)\le\Delta(S)\le r_k\delta(S).
\end{equation}
The constants are independent of the cardinality of $S$ and of the choice of
isotropic measure supported in $S$. For fixed $k$, the linear power is optimal.
There are admissible sets with $\delta(S)\downarrow0$ and
$\Delta(S)$ comparable to $\delta(S)$.
\end{theorem}

The lower constant obtained below is not explicit globally. A local linear
estimate is combined with compactness away from the extremizers. In particular,
\cref{thm:stability} is a fixed-dimension statement and does not assert uniform
control as $k$ tends to infinity.

Our second result concerns a matrix optimization problem. For integers $N\ge k$,
let
\begin{equation}\label{eq:matrix-class}
 \cR_{N,k}=\left\{R\in\R^{N\times N}:\begin{array}{l}
 R=R^{\mathsf T}\succeq0,\quad\diag R=\mathbf1,\\
 \spec(R)=\bigl((N/k)^{[k]},0^{[N-k]}\bigr)
 \end{array}\right\},
\end{equation}
where brackets denote multiplicity. This class is nonempty for every $N\ge k$.
Set
\begin{equation}\label{eq:L}
 L_{N,k}=\min_{R\in\cR_{N,k}}\E\max_{1\le i\le N}X_i,
 \qquad X\sim N(0,R).
\end{equation}
Nonemptiness follows from the real Schur--Horn theorem, and attainment
follows from the frame representation in \cref{sec:arithmetic}.

\begin{theorem}[Divisibility gap]\label{thm:arithmetic}
Fix $k\ge2$, and let $c_k$ be as in \cref{thm:stability}.
\begin{enumerate}
 \item If $N=mk$, then $L_{N,k}=a_k$. The minimizing matrices are exactly the
 matrices that, after a simultaneous permutation of rows and columns, consist
 of $k$ diagonal blocks equal to the $m\times m$ all-ones matrix.
 \item If $N\ge k$ and $k\nmid N$, then
 \begin{equation}\label{eq:arithmetic-lower}
  L_{N,k}-a_k\ge\frac{c_k}{\sqrt{kN}}.
 \end{equation}
 \item If $k\nmid N$ and $N\ge2k^2$, then
 \begin{equation}\label{eq:arithmetic-upper}
  L_{N,k}-a_k\le\frac{2\sqrt{k}\,r_k}{\sqrt N}
                  \le\frac{2k}{\sqrt N}.
 \end{equation}
\end{enumerate}
Consequently, along integers not divisible by $k$,
$L_{N,k}-a_k\asymp_k N^{-1/2}$ as $N\to\infty$.
\end{theorem}

The mechanism behind \cref{thm:arithmetic} is geometric and arithmetic. If a
frame is close to a basis, isotropy forces the total weight in each cluster to
equal $1$ up to a quadratic error in the cluster radius. For equal weights,
these masses are multiples of $k/N$. When $k\nmid N$, their discrepancy from
$1$ forces a cluster radius of order at least $N^{-1/2}$. Linear stability
converts this radius into the same order of excess Gaussian width.

In dimension two the perimeter formula permits a sharp expansion.

\begin{theorem}[Planar asymptotics]\label{thm:planar}
As odd integers $N$ tend to infinity,
\begin{equation}\label{eq:planar-main}
 L_{N,2}=\frac1{\sqrt\pi}+\frac1{2\sqrt{\pi N}}+O(N^{-1}).
\end{equation}
In particular,
$\lim_{m\to\infty}\sqrt{2m+1}(L_{2m+1,2}-a_2)=1/(2\sqrt\pi)$.
\end{theorem}

The construction proving this theorem has three distinct directions and
unequal splitting multiplicities. The symmetric construction used for
\cref{thm:arithmetic} has a strictly larger leading constant.
\Cref{cor:planar-rigidity} describes every asymptotically minimizing
planar sequence. For higher rank, \cref{sec:tangent} gives limiting
moment constraints and an attained finite-dimensional lower bound.

\subsection{Relation to earlier work}

The geometric Brascamp--Lieb inequality originates in the work of Brascamp and
Lieb \cite{BrascampLieb76} and its geometric normalization by Ball \cite{Ball89}.
Barthe's work \cite{Barthe98reverse} and Valdimarsson's characterization of
optimizers \cite{Valdimarsson08} provide equality theory. We use the stochastic
bound \eqref{eq:dom-intro} as a direct application of that inequality, rather
than as the principal contribution. The tail argument below gives a
self-contained equality proof without invoking the general classification of
optimizers.

Mean-width inequalities for isotropic measures with additional assumptions
have been studied extensively. Centered measures lead to simplex extremizers
\cite{Barthe98simplex,Schmuck99,LiLeng12}; quantitative results are given by
B\"or\"oczky, Fodor, and Hug \cite{BFH21}. For even isotropic measures, the
corresponding extremizer is a cross-polytope, and recent stability estimates
appear in \cite{BFH25}. Our admissibility condition imposes neither centering
nor evenness. The extremizing set has $k$ points, and the stability estimate is
in the Hausdorff distance between sets of directions. Thus the hypotheses and
extremizers differ from those in the centered and even problems.

Recent Gaussian comparison results of Mulgund are also relevant.
Under $R\succeq\mathbf1\mathbf1^{\mathsf T}/n$, the maximum of a centered Gaussian
vector with correlation matrix $R$ is stochastically dominated by the
maximum of $n$ independent standard normals \cite{Mulgund26weak}.
For arbitrary correlation matrices, the regular simplex gives a
stochastic upper bound \cite{Mulgund26simplex}. These comparisons address
upper extremality. Here we minimize the expected maximum under the fixed
spectrum in \eqref{eq:matrix-class}. Our isotropy assumption gives a lower
comparison with $k$ independent normals. The linear stability estimate,
divisibility gap, and planar coefficient quantify this lower extremum.

The frame-potential characterization of tight frames is classical
\cite{BenedettoFickus03}. On the present feasible class that potential is
constant, equal to $N^2/k$, so it does not distinguish the widths being
minimized. Prescribed-spectrum frame constructions \cite{CahillEtAl13}
give existence, while sparse energy-transfer methods such as Spectral
Tetris \cite{CasazzaWoodland14} provide related construction techniques.
Our construction controls the distance from a basis as well as the frame
operator. Gaussian comparison bounds in terms of covariance differences
\cite{CCK15} address a different constraint from fixing the spectrum.

The results developed here are the linear estimate \eqref{eq:stability},
the divisibility gap, and the sharp planar constant in
\eqref{eq:planar-main}. The cluster-center estimate removes first-order
motion of the cluster means. The planar proof then combines a perimeter
expansion with the variance forced by the odd multiplicity defect.

\section{Stochastic domination, equality, and compactness}\label{sec:preliminaries}

We first record the finite version of the basic comparison.

\begin{proposition}\label{prop:finite}
Let $u_i\in S^{k-1}$ and $d_i>0$, for $1\le i\le N$, satisfy
$\sum_i d_i u_i u_i^{\mathsf T}=I_k$. If
$M=\max_i\ip{g}{u_i}$, then
\begin{equation}\label{eq:finite-dom}
 \Prob(M\le t)\le\Phi(t)^k\qquad(t\in\R).
\end{equation}
Moreover, $\E M=a_k$ if and only if the distinct values among the $u_i$ form
an orthonormal basis. In that case, the weights on each basis vector sum to $1$.
\end{proposition}

\begin{proof}
For nonnegative integrable functions $f_i:\R\to[0,\infty)$, the geometric
Brascamp--Lieb inequality \cite{Ball89,Barthe98reverse} states that
\begin{equation}\label{eq:BL}
 \int_{\R^k}\prod_i f_i(\ip{x}{u_i})^{d_i}\,dx
 \le\prod_i\left(\int_\R f_i(s)\,ds\right)^{d_i}.
\end{equation}
Choose $f_i(s)=e^{-s^2/2}\ind_{(-\infty,t]}(s)$. Isotropy and
$\sum_i d_i=k$ turn the left side into
$(2\pi)^{k/2}\Prob(M\le t)$ and the right side into
$(2\pi)^{k/2}\Phi(t)^k$. This proves \eqref{eq:finite-dom}.

Let $F$ be the distribution function of $M$. Integrability and the layer-cake
identity give
\begin{equation}\label{eq:gap-integral}
 \E M-a_k=\int_\R\bigl(\Phi(t)^k-F(t)\bigr)\,dt.
\end{equation}
The integrand is nonnegative and right-continuous. If the expectations are
equal, it vanishes everywhere. A strictly positive value, by right-continuity,
would give a positive integral over an interval to its right. Thus equality of
expectations implies equality of the two distributions.

Let $v_1,\ldots,v_n$ be the distinct directions. For $j\ne l$ write
$\rho_{jl}=\ip{v_j}{v_l}<1$ and $A_j(t)=\{\ip{g}{v_j}>t\}$.
If $-1<\rho_{jl}<1$, then
\begin{equation}\label{eq:pair-tail}
 \Prob(A_j(t)\cap A_l(t))
 \le\Phibar\!\left(\sqrt{\frac{2}{1+\rho_{jl}}}\,t\right)
 =o(\Phibar(t))\qquad(t\to\infty).
\end{equation}
Indeed the intersection implies that the sum of the two normal variables is
larger than $2t$, while its variance is $2(1+\rho_{jl})$; the final comparison
follows from $\Phibar(t)\sim e^{-t^2/2}/(\sqrt{2\pi}\,t)$. If
$\rho_{jl}=-1$, the intersection is empty for $t\ge0$. Bonferroni's inequalities
and finiteness of $n$ now give
\[
 \Prob(M>t)=n\Phibar(t)(1+o(1)).
\]
Equality of laws gives the same expression with $n$ replaced by $k$, because
$1-\Phi(t)^k\sim k\Phibar(t)$. Hence $n=k$.

Group the weights on each distinct vector into $D_j>0$, and let $V$ have
columns $\sqrt{D_j}v_j$. Then $VV^{\mathsf T}=I_k$. Since $V$ is square,
$V^{\mathsf T}V=I_k$ as well. Its diagonal entries imply $D_j=1$, and its
off-diagonal entries imply that the $v_j$ are orthogonal. Conversely, an
orthonormal basis gives $k$ independent standard normal coordinates, and
repetitions do not change their maximum.
\end{proof}

\begin{proposition}\label{prop:compact-equality}
For every $S\in\cA_k$, \eqref{eq:dom-intro} holds. Moreover,
\begin{equation}\label{eq:equality}
 \Delta(S)=0\quad\Longleftrightarrow\quad S=Q\basis_k
 \text{ for some }Q\in O(k).
\end{equation}
\end{proposition}

\begin{proof}
If $\mu$ witnesses admissibility, then $\mu/k$ is a probability measure and
$I_k/k$ is the barycenter of $uu^{\mathsf T}$ under this measure. The image
of $S$ is compact, so its convex hull contains this barycenter. The symmetric
trace-one matrices lie in an affine space of dimension $k(k+1)/2-1$.
Carath\'eodory's theorem therefore supplies vectors $u_1,\ldots,u_s\in S$,
where $s\le k(k+1)/2$, and positive weights $d_i$ such that
$\sum_i d_i u_i u_i^{\mathsf T}=I_k$. Apply \cref{prop:finite} to
$T=\{u_1,\ldots,u_s\}$ and use $h_S\ge h_T$.

If $w(S)=a_k$, then $a_k\le w(T)\le w(S)$, so $T$ is an orthonormal basis.
If $u\in S\setminus T$, then
$\ip{u}{u}=1>\max_{v\in T}\ip{u}{v}$. By continuity the strict inequality
$h_S(x)>h_T(x)$ holds on an open neighborhood of $u$. This neighborhood has
positive Gaussian measure, and $h_S\ge h_T$ everywhere, contradicting
$w(S)=w(T)$. Thus $S=T$. The converse is immediate.
\end{proof}

\begin{lemma}\label{lem:compactness}
The class $\cA_k$ is compact in the Hausdorff metric. The functions $\delta$
and $w$ are continuous on it, and
\begin{equation}\label{eq:Lipschitz}
 |w(S)-w(T)|\le r_k d_H(S,T).
\end{equation}
For every $\eta>0$ such that
$\{S\in\cA_k:\delta(S)\ge\eta\}$ is nonempty, the minimum of $\Delta$
on this set is strictly positive.
\end{lemma}

\begin{proof}
The space of nonempty compact subsets of the compact sphere is compact in
$d_H$. To see that $\cA_k$ is closed, suppose $S_n\to S$ with isotropic
measures $\mu_n$ supported in $S_n$. All these measures have mass $k$.
After passing to a subsequence they converge weakly to a measure $\mu$ of
mass $k$. Since
\[
 \int\dist(u,S)\,d\mu_n(u)\le k d_H(S_n,S)\longrightarrow0,
\]
weak convergence shows that $\mu$ is supported in $S$. The coordinate
functions $u\mapsto u_i u_j$ are continuous, so their integrals retain
\eqref{eq:isotropy}. Thus $S\in\cA_k$.

For every $x$, the definition of $d_H$ implies
$|h_S(x)-h_T(x)|\le\norm{x}d_H(S,T)$. Taking expectations proves
\eqref{eq:Lipschitz}. The triangle inequality also gives
$|\delta(S)-\delta(T)|\le d_H(S,T)$. The final assertion follows from
compactness, continuity, and \cref{prop:compact-equality}.
\end{proof}

\section{Second-order control of cluster centers}\label{sec:clusters}

We next establish the geometric estimate used in the local stability proof.
Unlike a bound on arbitrary representatives, it places the weighted centers
within a quadratic error of an orthonormal basis.

\begin{lemma}[Cluster masses]\label{lem:masses}
Suppose $S\in\cA_k$ and $d_H(S,\basis_k)\le\varepsilon\le1/2$.
Then the sets
\[
 S_j=S\cap\{u:\norm{u-e_j}\le\varepsilon\},\qquad 1\le j\le k,
\]
are nonempty, pairwise disjoint compact sets with union $S$. For every
isotropic measure $\mu$ supported in $S$, their masses $D_j=\mu(S_j)$ satisfy
\begin{equation}\label{eq:mass-bound}
 1-k\varepsilon^2\le D_j\le1+2\varepsilon^2,
 \qquad |D_j-1|\le k\varepsilon^2.
\end{equation}
\end{lemma}

\begin{proof}
The covering and nonemptiness follow from the two parts of the Hausdorff
bound. The balls are disjoint because distinct basis vectors are at distance
$\sqrt2>2\varepsilon$. For $u\in S_j$,
$\ip{u}{e_j}=1-\norm{u-e_j}^2/2\ge1-\varepsilon^2/2$. For
$u\in S_l$ with $l\ne j$, $|\ip{u}{e_j}|\le\varepsilon$. Applying the
$j$th diagonal entry of \eqref{eq:isotropy} gives
\[
 D_j(1-\varepsilon^2/2)^2\le1\le D_j+k\varepsilon^2.
\]
For $0\le\varepsilon\le1/2$, one has
$(1-\varepsilon^2/2)^{-2}\le1+2\varepsilon^2$. The two bounds and $k\ge2$
give \eqref{eq:mass-bound}.
\end{proof}

\begin{lemma}[Cluster centers]\label{lem:centers}
Under the assumptions of \cref{lem:masses}, suppose also that
$k\varepsilon^2\le1/2$. Then $D_j>0$, and the centers
\[
 m_j=\frac1{D_j}\int_{S_j}u\,d\mu(u)
\]
admit an orthonormal basis $q_1,\ldots,q_k$ satisfying
\begin{equation}\label{eq:center-bound}
 \max_j\norm{m_j-q_j}\le2k\varepsilon^2.
\end{equation}
Furthermore,
\begin{equation}\label{eq:cluster-radius}
 \sup_{u\in S_j}\norm{u-q_j}\le2\varepsilon+2k\varepsilon^2.
\end{equation}
\end{lemma}

\begin{proof}
Positivity follows from \eqref{eq:mass-bound}. Define
\[
 C=\sum_j\int_{S_j}(u-m_j)(u-m_j)^{\mathsf T}\,d\mu(u),\qquad
 A=[\sqrt{D_1}m_1\ \cdots\ \sqrt{D_k}m_k].
\]
The variance decomposition gives
\begin{equation}\label{eq:variance}
 AA^{\mathsf T}=I_k-C,\qquad C\succeq0,\qquad
 \tr C\le\sum_j\int_{S_j}\norm{u-e_j}^2\,d\mu(u)
          \le k\varepsilon^2.
\end{equation}
The trace inequality uses the minimizing property of the mean for squared
distance. In particular, $\norm{C}_{\op}\le1/2$ and $A$ is invertible.
Its polar decomposition has the form
$A=(I_k-C)^{1/2}Q$ with $Q\in O(k)$. Since
$1-\sqrt{1-s}\le s$ for $0\le s\le1$, we obtain
\[
 \norm{A-Q}_{\op}\le\norm{C}_{\op}\le k\varepsilon^2.
\]
Let $q_j$ be the columns of $Q$. Using $\norm{m_j}\le1$ and
$|1-\sqrt{D_j}|\le|1-D_j|$, we find
\[
 \norm{m_j-q_j}
 \le\norm{\sqrt{D_j}m_j-q_j}+|1-\sqrt{D_j}|\norm{m_j}
 \le2k\varepsilon^2.
\]
Finally, $\norm{m_j-e_j}\le\varepsilon$ because $m_j$ is a barycenter of
points in the ball around $e_j$. The triangle inequality proves
\eqref{eq:cluster-radius}.
\end{proof}

\section{Proof of linear stability}\label{sec:stability}

The upper bound in \cref{thm:stability} follows from
\eqref{eq:Lipschitz}, since every orthonormal basis has width $a_k$.
We prove the lower bound locally before using compactness.

\begin{proposition}[Local estimate]\label{prop:local}
For each $k\ge2$ there are $\eta_k>0$ and $\beta_k>0$ such that
\begin{equation}\label{eq:local}
 \Delta(S)\ge\frac{\beta_k}{2}\delta(S)
 \qquad\text{if }S\in\cA_k\text{ and }\delta(S)\le\eta_k.
\end{equation}
\end{proposition}

\begin{proof}
Put $\varepsilon=\delta(S)$. If $\varepsilon=0$ there is nothing to prove.
Otherwise rotate a minimizing basis to $\basis_k$, and choose any isotropic
measure supported in $S$. Set
\[
 K_k=\sqrt{4+(k-1)/16},\qquad
 \varepsilon_{0,k}=\min\left\{\frac18,\frac1{2k},
                                \sqrt{\frac3{32kK_k}}\right\}.
\]
For the moment suppose $\varepsilon\le\varepsilon_{0,k}$, and use the
clusters, centers, and orthonormal basis from \cref{lem:masses,lem:centers}.
Since $2k\varepsilon^2\le\varepsilon$, \eqref{eq:cluster-radius} gives
$\norm{u-q_j}\le3\varepsilon$ for $u\in S_j$. This is smaller than
$\sqrt2/2$. Hence every point of $S_j$ is closer to $q_j$ than to any other
$q_l$. Nonemptiness of all clusters also implies
\begin{equation}\label{eq:new-basis-distance}
 d_H(S,\{q_1,\ldots,q_k\})
 =\max_j\max_{u\in S_j}\norm{u-q_j}\ge\varepsilon.
\end{equation}
The last inequality uses minimality of the original distance $\delta(S)$.
Choose $j$ and $u\in S_j$ with $\norm{u-q_j}\ge\varepsilon$.

Decompose $u-m_j=t+bq_j$, where $t\perp q_j$. Since $u$ and $q_j$ are unit,
\[
 \ip{u-q_j}{q_j}=-\tfrac12\norm{u-q_j}^2.
\]
Consequently, with $B_k=2k+9/2$,
\begin{equation}\label{eq:tangent}
 |b|\le B_k\varepsilon^2,\qquad
 \norm{t}\ge\varepsilon-(2k+B_k)\varepsilon^2.
\end{equation}
For the first estimate use $\norm{u-q_j}\le3\varepsilon$ and
\eqref{eq:center-bound}; for the second use
$\norm{u-q_j}\le\norm{t}+|b|+\norm{m_j-q_j}$.

Define $H(x)=\max_l\ip{x}{m_l}$. Since each $m_l\in\conv S$,
$h_S(x)\ge H(x)$ for every $x$. Consider the Gaussian event
\begin{equation}\label{eq:box}
 E_j=\{1\le\ip{g}{q_j}\le2,\quad
           |\ip{g}{q_l}|\le1/4\text{ for }l\ne j\}.
\end{equation}
On $E_j$, $\norm{g}\le K_k$ and, for $l\ne j$,
\[
 \ip{g}{m_j-m_l}\ge\frac34-4kK_k\varepsilon^2\ge\frac38.
\]
Thus $H(g)=\ip{g}{m_j}$ on $E_j$. We have both $h_S(g)\ge H(g)$ and
$h_S(g)\ge\ip{g}{u}$, so
\begin{equation}\label{eq:gain-pointwise}
 h_S(g)-H(g)\ge\ind_{E_j}(g)(\ip{g}{u-m_j})_+.
\end{equation}

Let $Z\sim N(0,I_{k-1})$, and define
\begin{equation}\label{eq:beta}
 \beta_k=\Prob(1\le g_1\le2)
  \min_{\substack{v\in\R^{k-1}\\\norm{v}=1}}
  \E\left[(\ip{Z}{v})_+\ind_{[-1/4,1/4]^{k-1}}(Z)\right].
\end{equation}
The expectation is continuous in $v$. For every unit $v$, the linear
functional $z\mapsto\ip{z}{v}$ is positive on an open subset of the cube,
where the Gaussian density is positive. Compactness of the unit sphere
therefore shows that $\beta_k>0$.

The coordinates of $g$ in the $q$-basis are independent standard normals.
Since $t\perp q_j$, the definition of $\beta_k$ and
$(x+y)_+\ge x_+-|y|$ give
\begin{equation}\label{eq:gain-expectation}
 \E\left[\ind_{E_j}(\ip{g}{u-m_j})_+\right]
 \ge\beta_k\norm{t}-2|b|.
\end{equation}
Here the term $2|b|$ bounds $|b\ip{g}{q_j}|$ on $E_j$; the estimate also
holds when $t=0$. On the other hand, \eqref{eq:center-bound} gives
\begin{equation}\label{eq:baseline-centers}
 \E H(g)\ge a_k-2kr_k\varepsilon^2.
\end{equation}
Combining \eqref{eq:tangent} and
\eqref{eq:gain-pointwise}--\eqref{eq:baseline-centers}, we obtain
\[
 \Delta(S)\ge\beta_k\varepsilon-A_k\varepsilon^2,
 \qquad A_k=2kr_k+\beta_k(2k+B_k)+2B_k.
\]
Taking $\eta_k=\min\{\varepsilon_{0,k},\beta_k/(2A_k)\}$ proves
\eqref{eq:local}.
\end{proof}

\begin{proofof}{Theorem {\ref{thm:stability}}}
The upper bound was established at the beginning of this section. For the
lower bound, use \cref{prop:local} when $\delta(S)\le\eta_k$. If the set
$\{S\in\cA_k:\delta(S)\ge\eta_k\}$ is nonempty, let $b_k>0$ be its
minimum excess, which exists by \cref{lem:compactness}. Since $\delta(S)\le2$
for subsets of the unit sphere, on this set
$\Delta(S)\ge b_k\ge(b_k/2)\delta(S)$. Thus one can take
$c_k=\min\{\beta_k/2,b_k/2\}$. If that set is empty, take $c_k=\beta_k/2$.
The optimality of the exponent follows from \cref{prop:sharpness} below.
\end{proofof}

\begin{proposition}[Optimality of the linear power]\label{prop:sharpness}
For every $k\ge2$ there is a family $S_\theta\in\cA_k$, $\theta\downarrow0$,
such that $\delta(S_\theta)\asymp\theta$ and
$\Delta(S_\theta)\asymp_k\theta$. In particular, no estimate
$\delta(S)\le C\Delta(S)^p$, with fixed $C$ and $p>1$, holds for all
admissible sets in any neighborhood of the orthonormal bases.
\end{proposition}

\begin{proof}
Let $R_\theta$ rotate the $(e_1,e_2)$-plane by $\theta$ and fix its orthogonal
complement, and put $S_\theta=\basis_k\cup R_\theta\basis_k$.
Assign weight $1/2$ to each vector of the two bases, adding weights at repeated
points. The resulting measure is isotropic. The midpoint basis gives
$\delta(S_\theta)\le2\sin(\theta/4)$.

For the reverse bound, take $0<\theta\le1/4$ and any basis at Hausdorff
distance $d<1/4$ from $S_\theta$. The points $e_1$ and $R_\theta e_1$ must
be assigned to the same basis vector. If they were within $d$ of distinct
basis vectors, the distance between those vectors would be at most
$2d+2\sin(\theta/2)<\sqrt2$, a contradiction. Hence
$d\ge\norm{e_1-R_\theta e_1}/2=\sin(\theta/2)$. If $d\ge1/4$, this
lower bound holds as well. Therefore
\[
 \sin(\theta/2)\le\delta(S_\theta)\le2\sin(\theta/4).
\]
The already proved two-sided estimate gives the claimed order for $\Delta$.
If $\delta\le C\Delta^p$ with $p>1$, the upper estimate
$\Delta\le r_k\delta$ would imply
$1\le Cr_k^p\delta^{p-1}$ along this family, which is impossible.
\end{proof}

\section{Flat covariance spectra and the divisibility gap}\label{sec:arithmetic}

\subsection{Frame representation and attainment}

For $N\ge k$, put $\lambda=N/k$ and define the set of unit-norm tight frames
\[
 \cF_{N,k}=\left\{(u_1,\ldots,u_N)\in(S^{k-1})^N:
                      \sum_{i=1}^Nu_i u_i^{\mathsf T}=\lambda I_k\right\}.
\]
If $U$ has rows $u_i^{\mathsf T}$, then $R=UU^{\mathsf T}$ has unit diagonal
and its nonzero eigenvalues are those of $U^{\mathsf T}U=\lambda I_k$.
Thus $R\in\cR_{N,k}$. Conversely, if $R\in\cR_{N,k}$, its spectral
factorization gives $R=UU^{\mathsf T}$ with $U^{\mathsf T}U=\lambda I_k$;
unit diagonal implies that its rows are unit vectors. Moreover,
\begin{equation}\label{eq:frame-law}
 (\ip{g}{u_1},\ldots,\ip{g}{u_N})\sim N(0,R).
\end{equation}
It follows that \eqref{eq:L} is the minimum of $w(\{u_i\})$ over
$\cF_{N,k}$. This set is closed in the compact space $(S^{k-1})^N$.
The objective is continuous, since changing every $u_i$ by at most $\epsilon$
changes it by at most $r_k\epsilon$. Hence a minimum exists whenever the
feasible set is nonempty. To see nonemptiness for every $N\ge k$, the
vector with $k$ entries $N/k$ and $N-k$ zeros majorizes the all-ones
vector in $\R^N$. Its first $j\le k$ entries sum to $jN/k\ge j$;
for $k<j\le N$ they sum to $N\ge j$, and both total sums are $N$.
The real Schur--Horn theorem therefore supplies a real symmetric matrix
with this spectrum and unit diagonal \cite{CahillEtAl13}. Its eigenvalues
are nonnegative, so the frame factorization above applies.

Assign the weight $d_i=1/\lambda$ to each frame vector. This makes the family
isotropic, so its width is at least $a_k$. By \cref{prop:finite}, equality
requires exactly $k$ distinct orthonormal vectors, with total weight $1$ at
each. If $n_j$ denotes the multiplicity of the $j$th vector, this says
$n_j/\lambda=1$. Consequently equality is possible exactly when $k\mid N$.
For $N=mk$, each multiplicity is $m$ and the Gram matrix has the stated
all-ones block form. This proves the first assertion of \cref{thm:arithmetic}.

\Needspace{5\baselineskip}
\subsection{The arithmetic lower bound}

\begin{lemma}\label{lem:integer}
If $k\nmid N$ and $(u_i)\in\cF_{N,k}$, then
\begin{equation}\label{eq:integer-radius}
 \delta(\{u_i\})\ge\frac1{\sqrt{kN}}.
\end{equation}
\end{lemma}

\begin{proof}
Write $\varepsilon=\delta(\{u_i\})$. If $\varepsilon>1/2$, the conclusion
follows from $kN\ge4$. Otherwise rotate a minimizing basis to $\basis_k$,
and use the clusters in \cref{lem:masses}. Count the indices in cluster $j$, including repetitions, as $n_j$.
Then $D_j=n_j/\lambda$, and \eqref{eq:mass-bound} gives
\[
 |n_j-\lambda|\le k\lambda\varepsilon^2=N\varepsilon^2.
\]
But $n_j$ is an integer and $\lambda=N/k$ is a noninteger with denominator
dividing $k$, so $|n_j-\lambda|\ge1/k$. This proves
\eqref{eq:integer-radius}.
\end{proof}

\Cref{thm:stability,lem:integer}, applied to every feasible frame, imply
\eqref{eq:arithmetic-lower}. Notice that the integer obstruction concerns
cluster masses, whose error is quadratic, whereas the width gap controls
the cluster radius linearly.

\subsection{A matching construction}

\begin{proposition}\label{prop:construction}
Suppose $N=mk+r$, $1\le r\le k-1$, and $m\ge2k$. There is a frame in
$\cF_{N,k}$ whose set of directions $S$ satisfies
\begin{equation}\label{eq:construction-bound}
 d_H(S,\basis_k)\le\frac{2\sqrt{k}}{\sqrt N},\qquad
 w(S)-a_k\le\frac{2\sqrt{k}\,r_k}{\sqrt N}.
\end{equation}
\end{proposition}

\begin{proof}
Let $p=k-r$, and set
\begin{equation}\label{eq:construction-parameters}
 P=\left\lfloor\frac{m+1}{2p}\right\rfloor,\qquad
 \sin^2\theta=\frac1{2kP},\qquad 0<\theta<\frac\pi2.
\end{equation}
Here $P\ge1$. Start with $m$ copies of $e_l$ for $1\le l\le p$, and
$m+1$ copies of $e_j$ for $p<j\le k$. For each pair $(j,l)$ with
$j>p$ and $l\le p$, replace $2P$ of the copies assigned to $j$ by $P$
copies of each of
\begin{equation}\label{eq:construction-vectors}
 v_{jl}^+=\cos\theta\,e_j+\sin\theta\,e_l,
 \qquad v_{jl}^-=\cos\theta\,e_j-\sin\theta\,e_l.
\end{equation}
Keep the remaining $m+1-2Pp$ copies of $e_j$. This number is nonnegative,
and the total number of vectors remains $pm+r(m+1)=N$. All vectors are unit.

The sum of the outer products for one group of $2P$ vectors is
\[
 P\bigl(v_{jl}^+(v_{jl}^+)^{\mathsf T}
       +v_{jl}^-(v_{jl}^-)^{\mathsf T}\bigr)
 =2P\cos^2\theta\,e_je_j^{\mathsf T}
   +2P\sin^2\theta\,e_le_l^{\mathsf T}.
\]
Compared with $2P$ copies of $e_j$, this transfers
$2P\sin^2\theta=1/k$ from the $j$th diagonal to the $l$th diagonal and
creates no off-diagonal entries. Each low-count axis receives $r/k$, and
each high-count axis loses $p/k$. The resulting diagonal entries are all
$m+r/k=N/k$. Hence the constructed family belongs to $\cF_{N,k}$.

Every vector assigned to axis $j$ lies within $2\sin(\theta/2)$ of $e_j$,
and each axis has at least one assigned vector. Thus the same quantity bounds
both directed Hausdorff distances. Since $\theta\in(0,\pi/2)$,
\[
 2\sin(\theta/2)\le\sqrt2\sin\theta=\frac1{\sqrt{kP}}.
\]
The inequality $\lfloor x\rfloor\ge x/2$ for $x\ge1$ gives
$P\ge(m+1)/(4p)$. Therefore
\[
 \frac1{\sqrt{kP}}
 \le2\sqrt{\frac{p}{k(m+1)}}
 \le\frac2{\sqrt{m+1}}
 \le\frac{2\sqrt{k}}{\sqrt N},
\]
where the last step uses $N<k(m+1)$. The width estimate follows from
\eqref{eq:Lipschitz} and $w(\basis_k)=a_k$.
\end{proof}

If $N\ge2k^2$ and $k\nmid N$, then $m=\lfloor N/k\rfloor\ge2k$.
Thus \cref{prop:construction} proves the third assertion of
\cref{thm:arithmetic}, completing its proof. The construction is used to
match the order of the minimum; we do not assert that it gives an exact
minimizer when $k\nmid N$.

\section{Recovery of covariance blocks}\label{sec:blocks}

The stability estimate controls near-minimizers. It bounds deviations
within a block and correlations between different blocks.

\begin{corollary}\label{cor:blocks}
Let $(u_i,d_i)_{i=1}^N$ be a finite isotropic family in $\R^k$, let
$R_{ii'}=\ip{u_i}{u_{i'}}$, and put
$\Delta=\E\max_i\ip{g}{u_i}-a_k$. If $\Delta\le c_k/2$, there are an
orthonormal basis $q_1,\ldots,q_k$ and a partition into nonempty classes
$I_1,\ldots,I_k$ such that
\begin{align}
 \max_{i\in I_j}\norm{u_i-q_j}&\le\frac{\Delta}{c_k},
                                                        \label{eq:block-directions}\\
 \left|\sum_{i\in I_j}d_i-1\right|&\le\frac{k\Delta^2}{c_k^2},
                                                        \label{eq:block-masses}\\
 0\le1-R_{ii'}&\le\frac{2\Delta^2}{c_k^2}
          &&(i,i'\in I_j),                              \label{eq:block-within}\\
 |R_{ii'}|&\le\frac{2\Delta}{c_k}
          &&(i\in I_j,\ i'\in I_l,\ j\ne l).           \label{eq:block-between}
\end{align}
If in addition $N=mk$ and all $d_i=1/m$, then $|I_j|=m$ for every $j$
whenever
\begin{equation}\label{eq:block-threshold}
 \Delta<\min\left\{\frac{c_k}{2},\frac{c_k}{\sqrt{km}}\right\}.
\end{equation}
\end{corollary}

\begin{proof}
Choose a basis realizing $\delta(S)$, where $S=\{u_i\}$.
\Cref{thm:stability} gives $\delta(S)\le\Delta/c_k\le1/2$.
Partition indices according to the clusters around that basis.
\Cref{lem:masses} proves \eqref{eq:block-directions} and
\eqref{eq:block-masses}. Within one cluster,
$\norm{u_i-u_{i'}}\le2\Delta/c_k$, so
$1-R_{ii'}=\norm{u_i-u_{i'}}^2/2$ proves \eqref{eq:block-within}.
For different clusters, orthogonality and the unit norms imply
\[
 |\ip{u_i}{u_{i'}}|
 =|\ip{u_i-q_j}{u_{i'}}+\ip{q_j}{u_{i'}-q_l}|
 \le\frac{2\Delta}{c_k},
\]
which is \eqref{eq:block-between}. In the equal-weight case,
\eqref{eq:block-masses} gives
$\bigl||I_j|-m\bigr|\le mk\Delta^2/c_k^2<1$ under
\eqref{eq:block-threshold}. This is an integer, and hence equals zero.
\end{proof}

\section{The planar asymptotic constant}\label{sec:planar}

In dimension two, a frame with vector angles $\theta_i$ is tight precisely
when $\sum_i e^{2\mathrm i\theta_i}=0$. For its set of directions $S$, the
planar perimeter formula gives
\begin{equation}\label{eq:perimeter}
 w(S)=\frac{P(\conv S)}{2\sqrt{2\pi}},\qquad a_2=\frac1{\sqrt\pi}.
\end{equation}
Here a segment has perimeter twice its length. The formula follows by
averaging the support function over directions and then over the Gaussian
radius, whose mean is $\sqrt{\pi/2}$.

\subsection{A uniform perimeter expansion}

\begin{lemma}\label{lem:planar-expansion}
Suppose a finite set of unit vectors consists of two nonempty clusters with
angles $x_i$ and $\pi/2+y_i$, where $|x_i|,|y_i|\le\epsilon$ and
$0<\epsilon\le1/8$. Write $x_-,x_+$ and $y_-,y_+$ for the extreme
angles of the two clusters. Then
\begin{equation}\label{eq:perimeter-expansion}
 \begin{split}
 P(\conv S)=2\sqrt2
 &+(x_+-x_-)+(y_+-y_-)\\
 &+\frac1{\sqrt2}(y_-+y_+-x_--x_+)+O(\epsilon^2).
 \end{split}
\end{equation}
The error constant is independent of the number of vectors and their
multiplicities.
\end{lemma}

\begin{proof}
Remove repeated points and order the remaining angles around the circle.
Each point is a vertex of the convex hull. A gap $v$ between consecutive
angles contributes the chord $2\sin(v/2)$ to the perimeter. Within the
first cluster the gaps sum to $x_+-x_-$, and their cubic powers sum to at
most $(x_+-x_-)^3$. Thus the total chord length within that cluster is
$x_+-x_-+O(\epsilon^3)$. The same argument applies to the second cluster.
The two remaining gaps are
$\pi/2+y_--x_+$ and $3\pi/2+x_--y_+$. Taylor expansion of their chord
lengths gives the constant $2\sqrt2$ and the linear term displayed in
\eqref{eq:perimeter-expansion}, with error $O(\epsilon^2)$. The calculation
also applies when a cluster consists of a single point.
\end{proof}

\begin{proofof}{Theorem {\ref{thm:planar}}}
We first prove a lower bound for a minimizing frame with odd size $N$.
\Cref{thm:arithmetic} gives an upper bound $O(N^{-1/2})$ for its excess.
\Cref{thm:stability} therefore places its directions within $O(N^{-1/2})$
of a basis. Apply \cref{lem:centers} and use its orthonormal basis as the
coordinate basis. The resulting cluster angles $x_i,y_i$ obey
\begin{equation}\label{eq:planar-small}
 \begin{gathered}
 \max_i(|x_i|,|y_i|)=O(N^{-1/2}),\quad n_1,n_2=N/2+O(1),\\
 \bar x=O(N^{-1}),\qquad \bar y=O(N^{-1}).
 \end{gathered}
\end{equation}
All constants here are uniform over minimizing frames. The angle means
follow from \eqref{eq:center-bound}, since the averages of $\sin x_i$
and $\sin y_i$ are $O(N^{-1})$ and $|s-\sin s|=O(|s|^3)$.
The count estimate follows from \cref{lem:masses}.

Let the centered, scaled angles in cluster $j$ range over $[-A_j,B_j]$,
and let $V_j$ be their empirical variance. Thus in the first cluster
the scaled angles are $\sqrt N(x_i-\bar x)$, and likewise in the second.
Their means are zero, so $A_j,B_j\ge0$ and
\begin{equation}\label{eq:range-variance}
 V_j\le A_jB_j.
\end{equation}
Indeed $(z+A_j)(B_j-z)\ge0$ on this interval, and its empirical mean is
$A_jB_j-V_j$.

The diagonal tight-frame equation is
$\sum_{i=1}^{n_1}\cos(2x_i)=\sum_{i=1}^{n_2}\cos(2y_i)$.
Using \eqref{eq:planar-small} and $\cos(2s)=1-2s^2+O(s^4)$ gives
\begin{equation}\label{eq:planar-defect}
 n_1-n_2=\frac{2n_1}{N}V_1-\frac{2n_2}{N}V_2+O(N^{-1})
         =V_1-V_2+O(N^{-1}).
\end{equation}
The integer $n_1-n_2$ is odd. Hence $|V_1-V_2|\ge1-O(N^{-1})$.

Put $\kappa_\pm=1\pm1/\sqrt2$. By
\cref{lem:planar-expansion}, replacing angle means by zero changes the
perimeter by only $O(N^{-1})$. Consequently
\begin{equation}\label{eq:planar-objective}
 \begin{gathered}
 \sqrt N\,(w(S)-a_2)=\frac{J}{2\sqrt{2\pi}}+O(N^{-1/2}),\\
 J=\kappa_+A_1+\kappa_-B_1+\kappa_-A_2+\kappa_+B_2.
 \end{gathered}
\end{equation}
Since $\kappa_-\kappa_+=1/2$, weighted arithmetic--geometric mean and
\eqref{eq:range-variance} imply
\[
 J\ge\sqrt2(\sqrt{V_1}+\sqrt{V_2})
   \ge\sqrt{2|V_1-V_2|}\ge\sqrt2-O(N^{-1}).
\]
Thus $L_{N,2}-a_2\ge1/(2\sqrt{\pi N})-O(N^{-1})$.

For the upper bound write $N=2m+1$, $n=m+1$, and set
\begin{equation}\label{eq:planar-construction}
 p=\frac{2-\sqrt2}{4},\quad A=\lfloor pn+1/2\rfloor,\quad B=n-A,
 \quad d=\arcsin\sqrt{\frac{N}{4AB}}.
\end{equation}
These expressions are well-defined for all sufficiently large $m$. Define
\[
 x=-\tfrac12\arg(A+B e^{2\mathrm i d}),\qquad y=x+d,
\]
using the argument near zero. Since
$|A+B e^{2\mathrm i d}|^2=n^2-4AB\sin^2d=m^2$, we have
\begin{equation}\label{eq:planar-exact}
 A e^{2\mathrm i x}+B e^{2\mathrm i y}=m.
\end{equation}
Take $m$ copies of $e_1$. Add $A$ copies of $(-\sin x,\cos x)$
and $B$ copies of $(-\sin y,\cos y)$. Equation \eqref{eq:planar-exact} says
that their doubled-angle sum vanishes. They form an exactly tight frame.

Since $A/n=p+O(N^{-1})$ and $p(1-p)=1/8$, expansion of
\eqref{eq:planar-construction} yields
\[
 \begin{gathered}
 \sqrt N\,d=2\sqrt2+O(N^{-1}),\\
 \sqrt N\,x=-(1+\sqrt2)+O(N^{-1}),\quad
 \sqrt N\,y=\sqrt2-1+O(N^{-1}).
 \end{gathered}
\]
For the last two identities use
$\arg(A+B e^{2\mathrm i d})=2(B/n)d+O(d^3)$.
The first cluster has zero range, while the second has the scaled endpoints
$-(1+\sqrt2)$ and $\sqrt2-1$ up to $O(N^{-1})$.
Their contribution to $J$ is
$\kappa_-(1+\sqrt2)+\kappa_+(\sqrt2-1)=\sqrt2$.
\Cref{lem:planar-expansion} gives the matching upper bound and proves
\eqref{eq:planar-main}.
\end{proofof}

\subsection{The geometry of asymptotic minimizers}

\begin{corollary}\label{cor:planar-rigidity}
Suppose $N=2m+1\to\infty$ and a sequence of tight frames satisfies
\[
 w(S_N)-a_2=\frac1{2\sqrt{\pi N}}+o(N^{-1/2}).
\]
After orthogonal transformations and a relabeling of the clusters, their
sizes are $m$ and $m+1$ for all sufficiently large $N$. In the coordinates
used in \eqref{eq:planar-small}, the scaled centered angles of the smaller
cluster converge uniformly to zero. The minimum and maximum scaled
centered angles of the larger cluster converge to $-(1+\sqrt2)$ and
$\sqrt2-1$. Its empirical angle distribution converges weakly to
\begin{equation}\label{eq:planar-limit-law}
 p\delta_{-(1+\sqrt2)}+(1-p)\delta_{\sqrt2-1},\qquad
 p=\frac{2-\sqrt2}{4}.
\end{equation}
\end{corollary}

\begin{proof}
All preceding uniform expansions apply to this sequence. Equation
\eqref{eq:planar-objective} gives $J\to\sqrt2$. The lower bound
$J\ge\sqrt{2|n_1-n_2|}+o(1)$ forces $|n_1-n_2|=1$ eventually, since
these differences are odd integers. Relabel the larger cluster as the
second one. Then $V_2-V_1\to1$ and
$\sqrt{V_1}+\sqrt{V_2}\le J/\sqrt2\to1$, so $V_1\to0$ and
$V_2\to1$.

The contribution of the second cluster to $J$ is at least
$\sqrt{2V_2}\to\sqrt2$. The first contribution therefore tends to
zero. Its coefficients are positive, so both of its endpoints tend to
zero. In the second cluster, equality in the bounds
\[
 (\kappa_-A_2+\kappa_+B_2)^2
 =2A_2B_2+(\kappa_-A_2-\kappa_+B_2)^2\ge2V_2
\]
is approached. This forces $A_2B_2\to1$ and
$\kappa_-A_2-\kappa_+B_2\to0$, giving the stated endpoints.
For a centered scaled angle $Y$, the empirical average of
$(Y+A_2)(B_2-Y)$ is $A_2B_2-V_2\to0$. This nonnegative expression is
bounded below away from zero on each closed subinterval strictly between
the limiting endpoints. Every weak limit is therefore supported on those
two endpoints. Its mean is zero, which fixes the masses in
\eqref{eq:planar-limit-law}.
\end{proof}

\begin{remark}
The symmetric construction in \cref{prop:construction} has, for $k=2$,
scaled endpoints $-1$ and $1$ in its larger cluster. Its excess is
$1/\sqrt{2\pi N}+O(N^{-1})$. It is therefore not asymptotically optimal.
The leading constant is $\sqrt2$ times the one in \cref{thm:planar}.
\end{remark}

The leading asymptotic geometry does not determine a unique finite frame.
For example, fix an integer $s\ge1$ and place $s$ copies at $e_2$. Split
the other $n'=m+1-s$ vectors in the larger cluster in proportions tending
to $p$ and $1-p$. In \eqref{eq:planar-construction}, replace $n$ by $n'$
and $N$ under the square root by $2n'-1$. The two split groups then have
doubled-angle sum $n'-1=m-s$. Adding the $s$ copies at $e_2$ restores
\eqref{eq:planar-exact}. The scaled endpoint limits are unchanged, so
these frames attain the same leading constant while having an additional
direction. No exact finite-$N$ minimizer is identified by this argument.

\section{Limiting constraints in higher rank}\label{sec:tangent}

The planar calculation uses that a centered real distribution is controlled
by its two endpoints. In higher rank the first-order width depends on the
whole tangent support. We record the constraints that follow rigorously
from near-minimality and state the remaining obstruction.

Fix $k$, $1\le r<k$, and $N=mk+r$. Consider tight frames with excess
at most $M N^{-1/2}$, for a fixed $M$. By \cref{thm:stability,lem:centers},
after an orthogonal transformation their clusters have the form
\begin{equation}\label{eq:tangent-parameter}
 u_i=\frac{e_j+h_i}{\sqrt{1+\norm{h_i}^2}},\qquad
 h_i\perp e_j,\quad i\in I_j,
\end{equation}
where, uniformly in the sequence,
\begin{equation}\label{eq:tangent-bounds}
 \begin{gathered}
 \max_i\norm{h_i}=O_{k,M}(N^{-1/2}),\quad
 \frac1{n_j}\sum_{i\in I_j}h_i=O_{k,M}(N^{-1}),\\
 n_j=m+s_j,\quad s_j=O_{k,M}(1),\quad \sum_j s_j=r.
 \end{gathered}
\end{equation}
The mean estimate follows from the cluster-center estimate and
$h_i-P_{e_j^\perp}u_i=O(\norm{h_i}^3)$.

Expanding the diagonal entries of the frame equation gives
\begin{equation}\label{eq:tangent-diagonal}
 n_j-N/k=\sum_{i\in I_j}\norm{h_i}^2
 -\sum_{l\ne j}\sum_{i\in I_l}h_{i,j}^2+O_{k,M}(N^{-1}).
\end{equation}
For $j\ne l$, its off-diagonal entries give
\begin{equation}\label{eq:tangent-offdiagonal}
 0=\sum_{i\in I_j}h_{i,l}+\sum_{i\in I_l}h_{i,j}
   +\sum_{a\ne j,l}\sum_{i\in I_a}h_{i,j}h_{i,l}
   +O_{k,M}(N^{-1/2}).
\end{equation}
These remainders follow from
$(1+\norm h^2)^{-1}=1-\norm h^2+O(\norm h^4)$ and
\eqref{eq:tangent-bounds}. The off-diagonal equations include second-order
center shifts and cannot be discarded in a construction of exact frames.

Let $K_{j,N}=\{\sqrt N h_i:i\in I_j\}$, and define
\begin{equation}\label{eq:tangent-functional}
 \Psi(K_1,\ldots,K_k)
 =\sum_{j=1}^k\E\left[\ind_{\{g_j>\max_{l\ne j}g_l\}}h_{K_j}(g)\right].
\end{equation}
Ties have Gaussian probability zero.

\begin{lemma}\label{lem:first-variation}
For the above sequence,
\begin{equation}\label{eq:first-variation}
 \sqrt N\,\Delta(S_N)
 =\Psi(K_{1,N},\ldots,K_{k,N})+O_{k,M}(N^{-1/2}).
\end{equation}
\end{lemma}

\begin{proof}
Put $t=N^{-1/2}$ and let $R$ bound all the tangent sets. Normalization in
\eqref{eq:tangent-parameter} changes each Gaussian coordinate from
$g_j+t\ip{g}{z}$ by at most $C_Rt^2\norm g$. Hence the maximum differs
from $\max_j(g_j+t h_{K_{j,N}}(g))$ by the same bound.

Let $J$ be the index of the largest $g_j$. If the perturbed maximum uses
another index, then for some $j\ne l$,
$|g_j-g_l|\le2tR\norm g$. The difference between the perturbed maximum
and $g_J+t h_{K_{J,N}}(g)$ lies between zero and $2tR\norm g$ and
vanishes outside this union of events. Write $g=\rho\omega$ with
$\omega$ uniform on the sphere and independent of $\rho$. For each
fixed unit vector $v$, the spherical measure of
$\{|\ip{\omega}{v}|\le a\}$ is $O_k(a)$ for small $a$. This follows
from the coordinate density, which is bounded near zero, also when
$k=2$. There are only $\binom k2$ pairs. Thus the expected difference
is $O_{k,R}(t^2)$, because $\E\rho=r_k<\infty$. Taking expectations
proves the result.
\end{proof}

\begin{proposition}\label{prop:tangent-limits}
Every such sequence has a subsequence along which all $s_j$ are constant,
$K_{j,N}$ converge in Hausdorff distance to compact sets
$K_j\subset e_j^\perp$, and their empirical measures converge weakly
to probability measures $\nu_j$ supported in $K_j$. These limits satisfy
\begin{equation}\label{eq:limiting-moments}
 \int z\,d\nu_j(z)=0,\qquad
 ks_j-r=\tr\Sigma_j-\sum_{l\ne j}(\Sigma_l)_{jj},\qquad
 \Sigma_j=\int zz^{\mathsf T}\,d\nu_j(z).
\end{equation}
Along this subsequence,
\begin{equation}\label{eq:limiting-objective}
 \sqrt N\,\Delta(S_N)\longrightarrow\Psi(K_1,\ldots,K_k).
\end{equation}
\end{proposition}

\begin{proof}
The integer shifts and tangent sets are uniformly bounded by
\eqref{eq:tangent-bounds}. Pass to subsequences using compactness of
bounded finite integer sets, the hyperspace of compact subsets of a
fixed ball, and probability measures on that ball. Weak limits remain
supported in the Hausdorff limits. The mean estimates give the first
identity in \eqref{eq:limiting-moments}. In
\eqref{eq:tangent-diagonal}, replace $h_i$ by $z_i/\sqrt N$ and use
$n_j/N\to1/k$. Bounded support makes the second moments continuous
under weak convergence, yielding the second identity. Finally, Hausdorff
convergence of the tangent sets gives convergence of their support
functions, dominated by $R\norm g$. Apply \cref{lem:first-variation}.
\end{proof}

There is a finite-dimensional relaxation of the limiting problem. Set
$D=k(k+1)/2$ and allow at most $D$ tangent points in each cluster, with
nonnegative probability weights, zero mean, and second moments satisfying
\eqref{eq:limiting-moments} for integers $s_j$ with sum $r$. Minimize
$\Psi$ over these data and denote the infimum by $\Gamma_{k,r}$.
This is a lower-bound problem, not an asserted formula for an asymptotic
constant.

\begin{proposition}\label{prop:relaxation}
The relaxation has a positive attained minimum and
\begin{equation}\label{eq:relaxation-bound}
 0<\Gamma_{k,r}\le
 \liminf_{m\to\infty}\sqrt{mk+r}\,(L_{mk+r,k}-a_k).
\end{equation}
\end{proposition}

\begin{proof}
For any limit in \cref{prop:tangent-limits}, apply Carath\'eodory's
theorem to $z\mapsto(z,zz^{\mathsf T})$ on $K_j\subset\R^{k-1}$.
Its target has dimension $(k-1)+(k-1)k/2=D-1$. Thus at most $D$ points
preserve its mean and second moments. Restricting the support to these
points can only decrease $\Psi$. Apply this reduction to a subsequence
of minimizing frames that realizes the liminf to obtain the upper
inequality in \eqref{eq:relaxation-bound}. This also proves that the
relaxation is feasible.

For each feasible cluster, zero mean implies $0\in\conv K_j$, so
$h_{K_j}\ge0$. On the Gaussian box \eqref{eq:box} with $q_j=e_j$,
the $j$th coordinate wins. For any $z\in K_j\subset e_j^\perp$,
$h_{K_j}(g)\ge(\ip{g}{z})_+$. The argument defining $\beta_k$ in
\eqref{eq:beta} gives
$\Psi\ge\beta_k\max_{j,z\in K_j}\norm z$.
Consequently sublevel sets have bounded tangent points. The moment
constraints then bound all integers $s_j$. Pad each cluster to $D$
points by repetitions with zero weights. Compactness of the points and
weights, and continuity of $\Psi$ and the constraints, give attainment.
If the minimum were zero, all tangent points would be zero, so
\eqref{eq:limiting-moments} would give $ks_j=r$ for every $j$. This
contradicts $0<r<k$.
\end{proof}

The relaxation omits an exact recovery theorem. One must still approximate
its weighted tangent configurations by integer multiplicities and correct
all frame equations with an error small enough to preserve the leading
width. Neither \cref{prop:tangent-limits} nor
\cref{prop:relaxation} proves that this can be done. In particular, they
do not prove convergence of the scaled minima for $k\ge3$, nor that
every optimal cluster size is $m$ or $m+1$. The integer shifts and tangent
supports are genuine variables in the relaxed problem.

The Paulsen problem concerns correction of nearly equal-norm Parseval
frames. Quantitative results such as \cite{HamiltonMoitra19} control a
sum of squared displacements. Applying them here would require an
additional estimate on the largest displacement after rescaling to unit
norm. Such a recovery estimate is not supplied by the present argument.

\section{Further questions}

The global lower constant $c_k$ in \cref{thm:stability} is obtained by
compactness. Determining an explicit bound with useful dependence on $k$
would strengthen the result for growing rank. A separate question is the
optimal local constant near an orthonormal basis, where the proof reduces
the problem to Gaussian gains on coordinate regions.

For $k\ge3$ and $1\le r<k$, convergence of the scaled minima
remains open. The relaxation
in \cref{prop:relaxation} gives a positive lower bound, but its sharpness
requires an exact recovery argument. For $k=2$, \cref{thm:planar}
settles the leading constant and the $O(N^{-1})$ remainder. An exact
finite-$N$ formula and a classification of exact minimizers remain open
in the present analysis.

Finally, approximate isotropy would require a version of the cluster-center
argument that tracks both the Gaussian excess and the error in the second
moment matrix. Such a result could extend block recovery to correlation
matrices whose nonzero eigenvalues are close to, rather than exactly at,
$N/k$. No such extension is asserted here.

\appendix
\section{Appendix}
\subsection{Computational diagnostics}\label{app:diagnostics}

The proofs above are analytic. The accompanying script \texttt{verify.py}
provides reproducible numerical diagnostics for the construction and the
cluster-center estimates; the output is recorded in \texttt{verification.\allowbreak json}.
It checks 264 instances of \eqref{eq:construction-vectors}, spanning
$2\le k\le12$, every nonzero residue modulo $k$, and four choices of $m$
for each pair $(k,r)$. These checks verify the row norms, multiplicities,
frame operator, and stated distance bound to floating-point accuracy.
There are also 270 checks of the cluster-center calculation, including
weighted unions of nearby orthonormal bases with a mixing weight of $10^{-10}$.

The planar script \texttt{verify\_planar.py} checks 28 asymmetric frames,
300 perimeter expansions, and 600 range-variance identities. The frame
checks include variants with central vectors. Results are saved in
\texttt{planar-\allowbreak verification.\allowbreak json}.
The planar widths in \cref{tab:examples} are computed from
\eqref{eq:perimeter}. They evaluate explicit frames, rather than the
unknown exact finite-$N$ minima.

\begin{table}[htbp]
\centering
\caption{Scaled excess $\sqrt N\,(w(S)-a_2)$ for the symmetric and
asymmetric planar constructions, evaluated in floating point.}\label{tab:examples}
 \begin{tabular}{rrr}
  \toprule
  $N$ & Symmetric & Asymmetric\\
  \midrule
  21     & 0.39498841 & 0.22854221\\
  201    & 0.39518184 & 0.26537954\\
  2001   & 0.39748935 & 0.27718280\\
  20001  & 0.39845605 & 0.28057974\\
  200001 & 0.39878583 & 0.28161977\\
  \midrule
  Limit & $1/\sqrt{2\pi}$ & $1/(2\sqrt\pi)$\\
  \bottomrule
 \end{tabular}
\end{table}

Additional examples in dimensions $3$ and $5$ use eight independent scrambled
Sobol runs of $65\,536$ points each, with the same Gaussian samples used for
the constructed maximum and the reference basis maximum. The reported
standard errors measure variation between randomized runs. These diagnostics
are not proofs of the theorems, of optimality of the construction, or of
publication novelty.

\subsection{Use of AI-assisted tools}
The author is responsible for the mathematical theory and arguments in this
manuscript. ChatGPT assisted with finding relevant sources and reviewing the
mathematics.

\end{document}